\documentclass[a4paper,12pt,final]{amsart}

\usepackage{latexsym,amsmath,amsthm,amssymb}
\usepackage{mathrsfs,dsfont}

\newtheorem{definition}{Definition}

\newtheorem{lemma}[definition]{Lemma}
\newtheorem{theorem}[definition]{Theorem}
\newtheorem{corollary}[definition]{Corollary}

\begin{document}

\title[Simplified Proof of the Theorem of Varopoulos \ldots]
{Simplified Proof of the Theorem of Varopoulos in the Commutative Case}

\author{Marco Thill}

\address{bd G.\,-\,D.\ Charlotte 53, \
                 L\ -\,1331 Luxembourg City, \
                 Europe}
                 
\email{math@pt.lu}



\thanks{Many thanks to Torben Maack Bisgaard for his excellent comments}

\begin{abstract}
We give continuity properties of bitraces on (possibly non-commutative)
Banach $*$-algebras based on the Closed Graph Theorem, leading to
a simplified proof of the Theorem of Varopoulos in the commutative case.
\end{abstract}

\maketitle

%%%%%%%%%%%%%%%%%%%%%%%

\section{Introduction}\label{intro}

The Theorem of Varopoulos \cite{Var} says that a positive linear
functional on a Banach $*$-algebra with continuous involution
and with a bounded one-sided approximate unit is continuous.
Its proof depends on two closely related Factorisation Theorems,
one by Paul J.\ Cohen, and the other by Nicholas Th.\ Varopoulos
himself. We quote them here:

\begin{theorem}[Paul J.\ Cohen \cite{Coh}]\label{Cohen}%
Let $A$ be a Banach algebra with a bounded one-sided approximate
unit. Then every element of $A$ is a product of two elements of $A$.
\end{theorem}

\begin{theorem}[Nicholas Th.\ Varopoulos \cite{Var}]%
Let $A$ be a Banach algebra with a bounded left approximate unit.
Then for a sequence $(a_n)$ in $A$ converging to $0$, there exists
a sequence $(b_n)$ in $A$ converging to $0$ as well as $c \in A$
with $a_n = c \,b_n$ for all $n$.
\end{theorem}

Proofs of the above two results can also be found in the following books:
\cite[\S\ 6]{Zel}, \cite[\S\ 11]{BD}, and \cite[Section 5.2]{Pal1}.

Satish Shirali \cite{Shi} extended the Theorem of Varopoulos to the case
of a Banach $*$-algebra with a bounded one-sided approximate unit (and
with a possibly discontinuous involution). (This was non-trivial also because
Varopoulos used the fact that a bounded left approximate unit in a Banach
$*$-algebra with continuous involution gives rise to a bounded right
approximate unit.) Proofs of the Theorem of Varopoulos and Shirali can
also be found in the books \cite[Theorem 11.3.7 pp.\ 1189--1190]{Pal2}
and \cite[\S\ 27]{Thi}.

We shall show how to prove the Theorem of Varopoulos  in the commutative
case without using the more technical Factorisation Theorem of Varopoulos.
This will depend on continuity properties of bitraces on (possibly
non-commutative) Banach $*$-algebras. These continuity properties are
based in a novel way on the Closed Graph Theorem.

The Factorisation Theorem of Varopoulos heavily depends on
additional bounds on the factors in Cohen's Factorisation Theorem.
These additional bounds are hard to establish, and are not needed
for our proof. Furthermore, a number of simpler proofs for Cohen's
Factorisation Theorem without these sharp bounds have been given.
See the paragraph preceding the statement of \cite[Theorem 5.2.2]{Pal1}
and the references given there.

\section{Terminology and Notation}\label{term}

We use the terminology of \cite[Chapter XV, Sections 6 \& 7]{Dieu}.
In the following, let $A$ be a $*$-algebra, and let
$\langle \cdot , \cdot \rangle$ be a \underline{positive Hilbert form} on $A$,
that is, a positive semi-definite sesquilinear form on $A$ such that
\[ \langle ab,c \rangle = \langle b,a^*c \rangle \]
holds for all $a,b \in A$. For example, any positive linear functional $\varphi$
on $A$ induces a positive Hilbert form on $A$ by putting
\[ \langle a,b \rangle := \varphi (b^*a) \qquad (a,b \in A). \]
See \cite[pp.\ 348--349]{Dieu}.

The \underline{isotropic subspace}
\[ I := \{ \,a \in A \,: \,\langle a,b \rangle = 0 \text{ for all } b \in A \,\} \]
is a left ideal. The quotient space
\[ \underline{A} := A / I \]
is a pre-Hilbert space with respect to the inner product
\[ \langle \underline{a} , \underline{b} \rangle := \langle a,b \rangle \qquad (a,b \in A), \]
where $\underline{a} \in \underline{A}$ shall denote the projection of $a \in A$.
Since $I$ is a left ideal, it follows that for $a \in A$ an operator $\pi (a)$ mapping
$\underline{b} \in \underline{A}$ to $\underline{ab} \in \underline{A}$ is well defined.
We denote the completion of $\underline{A}$ by $H$, which is a Hilbert space.

If $A$ is a Banach $*$-algebra, then the operators $\pi (a)$ $(a \in A)$
are bounded, and thus have a continuation to $H$, which we denote
by the same symbol. It follows that the \underline{GNS representation}
\begin{alignat*}{2}
\pi : &\ A &\ \to &\ B(H) \\
 &\ a &\ \mapsto &\ \pi (a)
\end{alignat*}
is bounded. For proofs, see for example \cite[Theorem 22.14 p.\ 83]{Thi}.

One says that $\langle \cdot , \cdot \rangle$ is a \underline{bitrace}, if furthermore
\[ \langle a,b \rangle = \langle b^*,a^* \rangle \]
holds for all $a,b \in A$. One then also has
\[ \langle ba,c \rangle = \langle b,ca^* \rangle \]
for all $a,b,c \in A$, and it is this property that we shall need in lemma \ref{lemma}.
For example, if $\langle \cdot , \cdot \rangle$ is induced by a positive linear functional
$\varphi$, then $\langle\cdot , \cdot \rangle$ is a bitrace if and only if
\[ \varphi (ab) = \varphi (ba) \]
holds for all $a,b \in A$. One then says that $\varphi$ is a \underline{trace}.
See \cite[pp.\ 356--358]{Dieu}.

As usual, let $A^{\,2} := \mathrm{span} \,\{ \,ab \in A : a, b \in A \,\}$.

The pair $( A , \langle \cdot , \cdot \rangle )$ is called a \underline{Hilbert algebra},
if it is a pre-Hilbert space with $A^{\,2}$ dense, if $\langle \cdot , \cdot \rangle$ is a
bitrace, and if the operators $\pi (a)$ $(a \in A)$ all are bounded, cf.\ \cite[p.\ 358]{Dieu}.

In the following, we shall let $A$ be a Banach $*$-algebra, and let $p$ denote the
projection on the closed subspace $M$ of $H$ given by
\[ M := \overline{\mathrm{span}} \ \{ \,\pi (a) \,x \in H \,: \,x \in H,\, a \in A \,\} \subset H. \]

\section{The Proofs}

\begin{lemma}\label{lemma}%
Let $\langle \cdot , \cdot \rangle$ be a bitrace on a Banach $*$-algebra $A$.
The mapping 
\begin{alignat*}{2}
 & A &\ \to &\ H \\
 & a &\ \mapsto &\ p \,\underline{a}
\end{alignat*}
(with notation as above) is bounded.
\end{lemma}

\begin{proof}
We shall show that the mapping in question has closed graph.
So let $a_n \to 0$ in $A$, and $p \,\underline{a_n} \to x$ in $H$.
We have to show that $x = 0$. The vector $x$ lies in $M$. It is
thus enough to prove that $\langle \,x , \pi (b) \,\underline{c} \,\rangle =  0$
for all $b, c \in A$.
One calculates
\begin{align*}
& \,\langle \,x , \pi (b) \,\underline{c} \,\rangle
    = \lim _{n \to \infty} \langle \,p \,\underline{a_n} \,, \pi (b) \,\underline{c} \,\rangle
    = \lim _{n \to \infty} \langle \,\underline{a_n} \,, p \,\pi (b) \,\underline{c} \,\rangle \\
= & \lim _{n \to \infty} \langle \,\underline{a_n} \,, \pi (b) \,\underline{c} \,\rangle
    = \lim _{n \to \infty} \langle \,a_n \,, b \,c \,\rangle
    = \lim _{n \to \infty} \langle \,a_n \,c^*, b \,\rangle \\
= & \lim _{n \to \infty} \langle \,\pi (a_n) \,\underline{c^*} \,, \underline{b} \,\rangle  = 0
\end{align*}
by continuity of $\pi$.
\end{proof}

A positive Hilbert form $\langle \cdot , \cdot \rangle$ on a normed $*$-algebra
$(A, | \cdot |)$ is called \underline{bounded} if there exists $\gamma \geq 0$
such that
\[ | \,\langle a , b \rangle \,| \leq \gamma \cdot | \,a \,| \cdot | \,b \,| \]
holds for all $a,b \in A$.

\begin{theorem}
A bitrace $\langle \cdot , \cdot \rangle$ on a Banach $*$-algebra $A$ is
bounded if the corresponding GNS representation is non-degenerate.
The last condition is satisfied, if $( A , \langle \cdot , \cdot \rangle )$ is a
Hilbert algebra.
\end{theorem}

From the lemma also follows:

\begin{theorem}
A bitrace on a Banach $*$-algebra $A$ is bounded on $A^{\,2}$.
\end{theorem}

\begin{theorem}\label{preVarop1}%
A bitrace on a Banach $*$-algebra with a bounded one-sided
approximate unit is bounded.
\end{theorem}

\begin{proof}
This follows now from the Factorisation Theorem of Cohen
\linebreak (theorem \ref{Cohen} above).
\end{proof}

\begin{theorem}\label{preVarop2}%
Let $A$ be a Banach $*$-algebra with a bounded right \linebreak
approximate unit $(e_i)_{i \in I}$. Let $\langle \cdot , \cdot \rangle$
be a non-zero bitrace on $A$. The corresponding GNS
representation $\pi$ then is cyclic. Furthermore, there exists a cyclic
vector $c$ for $\pi$, such that the positive linear functional
$a \mapsto \langle \pi (a) c, c \rangle$ $(a \in A)$ induces the bitrace
$\langle \cdot , \cdot \rangle$. The vector $c$ can be chosen to be
any adherence point of the net $(\underline{e_i})_{i \in I}$ in
the weak topology.
\end{theorem}

(We give the proof here for the sake of staying self-contained.
A more general result is given in \cite[Theorem 27.4]{Thi}.
See also \cite[Theorem 27.5]{Thi}.)

\begin{proof}
By boundedness of $\langle \cdot , \cdot \rangle$ (theorem
\ref{preVarop1}), the net $(\underline{e_i})_{i \in I}$ is bounded
in $H$, and therefore has an adherence point in the weak topology
on $H$. Let $c$ be any such adherence point. By going over to a
subnet, we can assume that $(\underline{e_i})_{i \in I}$ converges
weakly to $c$. It shall be shown that $\pi (a) c = \underline{a}$
for all $a \in A$, which implies that $c$ is cyclic for $\pi$. It is
enough to show that for $b \in A$, one has
$\langle \pi (a)c, \underline{b} \rangle
= \langle \underline{a}, \underline{b} \rangle$. So one calculates
\begin{align*}
 & \langle \pi (a)c, \underline{b} \rangle
= \langle c, \pi (a)^* \underline{b} \rangle
= \lim _{i \in I} \,\langle \underline{e_i}, \pi (a)^* \underline{b} \rangle \\
= \,&\lim _{i \in I} \,\langle \pi (a)\underline{e_i}, \underline{b} \rangle
= \lim _{i \in I} \,\langle a e_i, b \rangle
= \langle a, b \rangle = \langle \underline{a}, \underline {b} \rangle,
\end{align*}
where we have used the continuity of $\langle \cdot, \cdot \rangle$.
Consider now the positive linear functional $\varphi$ on $A$ defined by
\[ \varphi (a) := \langle \pi (a) c, c \rangle \qquad (a \in A). \]
We have
\[ \varphi (a^*a) = \langle \pi (a^*a) c, c \rangle = {\| \,\pi (a)c \,\| \,}^2
= {\| \,\underline{a} \,\| \,}^2
= \langle \underline{a}, \underline{a} \rangle
= \langle a, a \rangle \]
whence, after polarisation,
\[ \langle a, b \rangle = \varphi(b^*a) \qquad (a, b \in A), \]
that is, $\varphi$ induces the bitrace $\langle \cdot , \cdot \rangle$.
\end{proof}

A positive linear functional $\varphi$ on a normed $*$-algebra
$(A, | \cdot |)$ is called \underline{representable}, if there exists
a non-degenerate continuous representation $\pi$ of $(A, | \cdot |)$
in a Hilbert space, and a vector $c$ such that
\[ \varphi (a) = \langle \pi(a) c, c \rangle \]
holds for all $a \in A$. The functional $\varphi$ then is continuous
by the con\-tinuity assumption on $\pi$.

It is known that a positive linear functional $\varphi$ on a
Banach $*$-algebra $A$ is representable, if and only if it
has \underline{finite variation}, in the sense that there exists
$\gamma \geq 0$ such that
\[ {| \,\varphi (a) \,| \,}^2 \leq \gamma \,\varphi (a^*a) \]
holds for all $a \in A$. See for example \cite[\S\S\ 24--25]{Thi}.

\begin{theorem}
A trace on a Banach $*$-algebra with a bounded one-sided
approximate unit has finite variation. Hence it is representable,
and thus also continuous.
\end{theorem}

\begin{proof}
Let $\varphi$ be a trace on a Banach $*$-algebra $A$ with a
bounded one-sided approximate unit. In case of a bounded
right approximate unit, it follows from theorem \ref{preVarop2}
and the Factorisation Theorem of Cohen (theorem \ref{Cohen}
above), that $\varphi$ is representable. Then $\varphi$ has 
finite variation and is Hermitian in the sense that
$\varphi (a^*) = \overline{\varphi (a)}$ for all $a \in A$. In case
of a bounded left approximate unit, one changes the multiplication
in $A$ to $(a,b) \mapsto ba$, and one applies the preceding
case of a bounded right approximate unit. Using the fact that
$\varphi$ turns out to be Hermitian, it follows again that
$\varphi$ has finite variation.
\end{proof}

\begin{corollary}
A positive linear functional on a commutative Banach \linebreak
$*$-algebra with a bounded one-sided approximate unit has
finite variation. Hence it is representable, and thus also continuous.
\end{corollary}

\end{document}